\documentclass[12pt]{amsart}

\usepackage[english]{babel}
\usepackage{mathrsfs, mathtools}
\usepackage{dsfont}
\usepackage{url}            
\usepackage{booktabs}       
\usepackage{amsfonts}       
\usepackage{nicefrac}       
\usepackage{microtype}      
\usepackage{lipsum}
\usepackage{enumitem}

\usepackage{a4wide}
\usepackage{color}
\usepackage{amssymb, amsmath}
\usepackage{mathrsfs}

\usepackage{float}
\usepackage{graphicx}

\usepackage{tikz}
\usepackage{tikz-cd}
\usetikzlibrary{snakes}
\usetikzlibrary{intersections, calc}
\usepackage{epsfig}
\usepackage[all]{xy}
\usepackage{epstopdf}
\usepackage{framed}

\usepackage{tikz-3dplot} 
\usepackage{subcaption} 

\newcounter{stmcounter}[section]

\newcounter{thmMaincounter}

\newtheorem{theorem}{Theorem}[section]
\newtheorem{lemma}[theorem]{Lemma}
\newtheorem{proposition}[theorem]{Proposition}
\newtheorem{corollary}[theorem]{Corollary}

\newtheorem{problem}[theorem]{Problem}
\newtheorem{theoremM}[thmMaincounter]{Theorem}

\theoremstyle{definition}
\newtheorem{definition}[theorem]{Definition}
\newtheorem{example}[theorem]{Example}

\theoremstyle{remark}
\newtheorem{remark}[theorem]{Remark}

\numberwithin{equation}{section}

\newcommand{\C}{{\mathbb{C}}}

\newcommand{\Z}{{\mathbb{Z}}}

\renewcommand{\P}{{\mathbb{P}}}

\newcommand{\Ro}{{\mathbb{R}}}

\newcommand{\lb}{\lbrace}
\newcommand{\rb}{\rbrace}
\renewcommand{\phi}{\varphi}

\newcommand{\algt}{\mathfrak{t}}

\DeclareMathOperator{\str}{star}

\DeclareMathOperator{\Hm}{Hom}

\begin{document}

\title[Borel-Hirzebruch type formula for GKM graph]{Borel-Hirzebruch type formula for the graph equivariant cohomology of a projective bundle over a GKM-graph}

\author[S.\ Kuroki]{Shintar\^o \textsc{Kuroki}}
\address{Department of Applied Mathematics \\
 Faculty of Science, Okayama University of Science \\
1-1 Ridai-Cho Kita-Ku Okayama-shi \\
Okayama 700-0005, Okayama, Japan}
\email{kuroki@ous.ac.jp}

\author[G.\ Solomadin]{Grigory \textsc{Solomadin}}
\address{
IRMA \\
University of Strasbourg, France}
\email{grigory.solomadin@gmail.com}




\thanks{The first author was partially supported by JSPS KAKENHI Grant Number 21K03262. The second author was partly supported by RFBR grant, project number $20-01-00675$ A, and by the contest "Young Russian Mathematics".
This work was partially supported by the Research Institute for Mathematical Sciences an International Joint Usage/Research Center located in Kyoto University. }

\begin{abstract}
In this paper, we introduce the GKM theoretical counterpart of the equivariant complex vector bundles as the ``leg bundle''.
We also provide a definition for the projectivization of a leg bundle and prove
the Borel-Hirzebruch type formula for its graph equivariant cohomology, assuming that
the projectivization is again a GKM graph.
Furthermore, we study the realization of the projective GKM fiber bundle, in the sense of Guillemin-Sabatini-Zara, can be obtained from the projectivization of a leg bundle.
\end{abstract}

\maketitle

\section{Introduction}
\label{sect:1}

A \textit{GKM manifold} is defined as an equivariantly formal manifold $M^{2m}$ equipped with an action of a compact torus $T^n:=(S^1)^n$.
This action satisfies the condition that the set of $0$ and $1$-dimensional orbits forms the structure of a graph.
This particular class of manifolds was originally introduced by Goresky-Kottwitz-MacPherson in \cite{GKM}.
In \cite{GKM}, they provide a detailed description of
the $T$-equivariant cohomology ring of $M$ using the associated graph.
Furthermore, by Guillemin-Zara in \cite{GZ},
the concept of a GKM graph is introduced as an abstract graph with edges labeled by vectors in the dual of Lie algebra of $T^{n}$.
These labeled graphs have certain inherent properties that are analogous
to those observed in GKM manifolds.

The corresponding GKM graphs provide a useful tool for studying various properties of GKM manifolds, see e.g.~\cite{GZ, GKZ, GHZ, Ku16}.
On the other hand, abstract GKM graphs themselves are fascinating objects that have garnered attention beyond their geometric motivations, see e.g.~\cite{FY, FIM, Ku19, KU, MMP, S, Y}.
In \cite{KU}, the concept of a {\it GKM graph with legs} (non-compact edges) is introduced as a result of studying the combinatorial generalization of toric hyperK${\rm\ddot{a}}$hler manifolds,
which were previously studied in \cite{HP}.
In particular, the GKM graph with legs defined in \cite{KU} includes
the combinatorial counterpart of the cotangent bundle of $\mathbb{C}P^{n}$
with the extended $T^{n}$-action derived from the $T^{n}$-action on $\mathbb{C}P^{n}$.
However, the GKM graphs with legs in \cite{KU} do not encompass the counterparts of all equivariant vector bundles over GKM manifolds.
Therefore, in this paper, we introduce the concept of a {\it leg bundle} as the combinatorial counterpart to any torus-equivariant complex vector bundle over a GKM manifold.
We then study its properties and characteristics.

The projectivizations of
torus-equivariant vector bundles (with an effective torus action) over GKM manifolds constitute another interesting class of spaces endowed with torus actions.
For example, the well-known results such as the Leray-Hirsch theorem (see \cite[Theorem 4D.1, p.~432]{H}) and the Borel-Hirzebruch formula (see \cite[Chapter V, Section 15]{BH}) describe the cohomology module and algebra (respectively) of the projectivization of a complex vector bundle.
In a related vein, the equivariant Leray-Hirsch theorem for a $T$-equivariant fiber bundle (with both the base and fiber being GKM manifolds) was
establised from a GKM perspective in \cite{GSZ'}.
Motivated by this, 
this paper focuses on the concept of {\it projectivization} of a leg bundle and proving a {\it Borel-Hirzebruch type formula} for its graph equivariant cohomology ring.
The following theorem serves as the first main result of this paper.
\begin{theoremM}[Theorem~\ref{main_theorem2}]
\label{main_theorem_intro}
Let $\xi$ be a rank $r+1$ leg bundle over a GKM graph $\Gamma$.
Assume that its projectivization $\Pi(\xi)$ is a GKM graph.
Then, there is the following isomorphism of  $H^*(\Gamma)$-algebras:
\begin{align*}
H^{*}(\Gamma)[\kappa]\big/\biggl( \sum_{s=0}^{r+1} (-1)^{s}c_{s}^{T}(\xi) \cdot \kappa^{r+1-s} \biggr)\simeq H^{*}(\Pi(\xi)),\quad \quad \kappa\mapsto c_{\xi}.
\end{align*}
\end{theoremM}
In general, both torus-equivariant vector bundles and their projectivizations do not possess the property of pairwise linear independence around fixed points.
Therefore, it is necessary to impose an additional assumption in Theorem~\ref{main_theorem_intro}.

In addition, in Section~\ref{sect:6}, we also study the realization problem of the projective bundle from the vector bundle.
If we add some conditions for the projective bundle, for example, the condition for the connection or the condition for the axial functions, then we can solve the realization problem.
 
The organization of this paper is as follows.
In Section~\ref{sect:2}, we define a leg bundle $\xi$ over a GKM graph $\Gamma$.
In Section~\ref{sect:3}, the projectivization of $\Pi(\xi)$ of a leg bundle $\xi$ is introduced.
In Section~\ref{sect:4}, 
the Chern class $c_{s}^{T}(\xi)$ and the tautological class $c_{\xi}$ of $\xi$ are recalled for further use in the subsequent part of this note.
In Section~\ref{sect:5}, we prove Theorem~\ref{main_theorem_intro}, i.e., the Borel-Hirzebruch type formula for the graph equivariant cohomology.
In the final section, Section~\ref{sect:6}, we discuss relation between projective and projectivization topological bundles over GKM-manifolds from combinatorial and topological point of view.

\section{Leg bundle over a GKM graph}
\label{sect:2}

The aim of this section is to define a leg bundle over the GKM graph which is a combinatorial counterpart of the equivariant complex vector bundle over a GKM manifold.

\subsection{Leg bundle over an abstract graph}
\label{sect:2.1}

Let $\mathcal{V}$ be a set of vertices, and $\mathcal{E}$ be a set of (oriented and possibly multiple) edges in $G$.
We denote $G=(\mathcal{V},\mathcal{E})$.
Throughout this paper, we assume that every graph $G$ is connected and finite.
We use the following notations:
\begin{itemize}
\item for the finite set $X$, the symbol $|X|$ represents its cardinality;
\item $i(e)\in \mathcal{V}$ is the initial vertex for $e\in \mathcal{E}$;
\item $t(e)\in \mathcal{V}$ is the terminal vertex for $e\in \mathcal{E}$;
\item $\overline{e}\in \mathcal{E}$ is the opposite directed graph of $e\in \mathcal{E}$;
\item $\str_{G}(p):=\{e\in \mathcal{E}\ |\ i(e)=p\}$ is the set of out-going edges from $p\in \mathcal{V}$.
\end{itemize}
The graph $G=(\mathcal{V},\mathcal{E})$ is called a {\it (regular) $m$-valent graph} if $|\str_{G}(p)|=m$ for every $p\in \mathcal{V}$.

\begin{definition}
Let $G=(\mathcal{V},\mathcal{E})$ be a graph.
The following pair of sets is called a {\it rank $r$ leg bundle} over $G$:
\begin{align*}
[r]_{G}:=
(\mathcal{V}, \mathcal{E}\sqcup \mathcal{V}\times [r]),
\end{align*}
where $[r]:=\{1,\ldots,r\}$.
An element $(p,j)\in \mathcal{V}\times [r]$ is called a {\it leg} of $[r]_{G}$ over $p\in \mathcal{V}$.
The set of legs over $p$, i.e., $[r]_{p}:=\{(p,1),\ldots , (p,r)\}$ is called the {\it fiber} of $[r]_{G}$ over $p$.
\end{definition}

The rank $r$ leg bundle $[r]_{G}$ over $G$ may be regarded as the Cartesian product of the graph $G$ and the $1$-vertex graph having $r$ non-compact edges, called {\it legs}, see Figure~\ref{abstract_leg_bdl}.

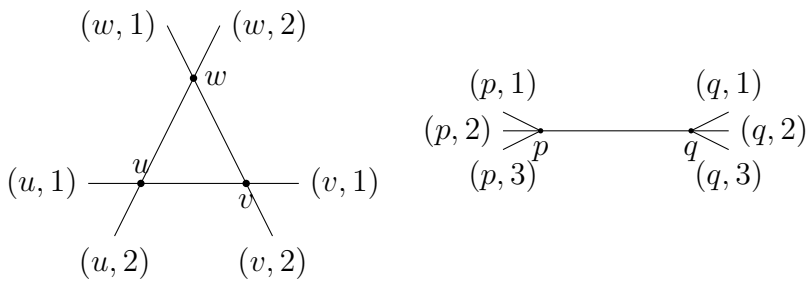
\begin{figure}[H]
\begin{center}
\begin{tikzpicture}
\begin{scope}[xscale=0.7, yscale=0.7]

\fill (-3,1) coordinate (w) circle (2pt);
\fill (-4,-1) coordinate (u) circle (2pt);
\fill (-2,-1) coordinate (v) circle (2pt);

\node[right] at (-3,1) {$w$};
\node[above] at (-4,-1) {$u$};
\node[below] at (-2,-1) {$v$};

\draw (u)--(v);
\draw (v)--(w);
\draw (w)--(u);

\draw (u)--(-5,-1);
\node[left] at (-5,-1) {$(u,1)$};
\draw (u)--(-4.5,-2);
\node[below] at (-4.5,-2) {$(u,2)$};

\draw (v)--(-1,-1);
\node[right] at (-1,-1) {$(v,1)$};
\draw (v)--(-1.5,-2);
\node[below] at (-1.5,-2) {$(v,2)$};

\draw (w)--(-3.5,2);
\node[left] at (-3.5,2) {$(w,1)$};
\draw (w)--(-2.5,2);
\node[right] at (-2.5,2) {$(w,2)$};

\end{scope}

\begin{scope}[xshift=100, xscale=0.5, yscale=0.5]
\fill (2,0) coordinate (q) circle (2pt);
\fill (-2,0) coordinate (p) circle (2pt);

\node[below] at (2,0) {$q$};
\node[below] at (-2,0) {$p$};

\draw (p)--(q);

\draw (p)--(-3,0);
\node[left] at (-3,0) {$(p,2)$};
\draw (p)--(-3,0.5);
\node[above] at (-3,0.5) {$(p,1)$};
\draw (p)--(-3,-0.5);
\node[below] at (-3,-0.5) {$(p,3)$};

\draw (q)--(3,0);
\node[right] at (3,0) {$(q,2)$};
\draw (q)--(3,0.5);
\node[above] at (3,0.5) {$(q,1)$};
\draw (q)--(3,-0.5);
\node[below] at (3,-0.5) {$(q,3)$};
\end{scope}
\end{tikzpicture}
\caption{The rank $2$ leg bundle over the triangle (left) and the rank $3$ leg bundle over the edge (right).}
\label{abstract_leg_bdl}
\end{center}
\end{figure}

\subsection{Leg bundle over a GKM graph}
\label{sect:2.2}

Let $G=(\mathcal{V}, \mathcal{E})$ be an $m$-valent graph.
We first recall the definition of a GKM graph $(G,\alpha,\nabla)$ which is originally defined in  \cite{GZ} (see e.g.~\cite{MMP, DKS} for a more general setting).
In this paper, we often use the following identification:
\begin{align*}
\mathbb{Z}^{n}\simeq (\mathfrak{t}^{n}_{\mathbb{Z}})^{*}\simeq {\rm Hom}(T^{n},S^{1})\simeq H^{2}(BT^{n})\subset H^{*}(BT^{n})\simeq \mathbb{Z}[x_{1},\ldots, x_{n}],
\end{align*}
where $(\mathfrak{t}^{n}_{\mathbb{Z}})^{*}$ is the lattice of $1$-parametric subgroups of $T^{n}$ and $\deg x_{i}=2$.

For $n\le m$,
any function $\alpha:\mathcal{E}\to (\mathfrak{t}^{n}_{\mathbb{Z}})^{*}$ satifying the following conditions (1)--(3) is called an {\it axial function}:
\begin{enumerate}
\item $\alpha(e)=\pm \alpha(\overline{e})$ for every edge $e\in\mathcal{E}$;
\item any two distinct elements in $\alpha(\str_{G}(p))=\{\alpha(e)\in (\mathfrak{t}^{n}_{\mathbb{Z}})^{*}\ |\ e\in \str_{G}(p)\}$ are linearly independent, i.e., \textit{pairwise linearly independent} (or {\it $2$-independent} for short), for every $p\in \mathcal{V}$; \item there is a bijection $\nabla_{e}:\str_{G}(i(e))\to \str_{G}(t(e))$ for every $e\in \mathcal{E}$ such that
\begin{enumerate}
\item $\nabla_{\overline{e}}=\nabla_{e}^{-1}$;
\item $\nabla_{e}(e)=\overline{e}$;
\item $\alpha(\nabla_{e}(e'))-\alpha(e')\equiv 0\mod \alpha(e)$ for every $e, e'\in \str_{G}(p)$.
\end{enumerate}
\end{enumerate}
The condition (3)-(c) is called a {\it congruence relation} on $e\in \mathcal{E}$.
The collection $\nabla=\{\nabla_{e}\ |\ e\in \mathcal{E}\}$ is called a {\it connection} on $(\Gamma,\alpha)$, and the bijection
$\nabla_{e}$ is also called a {\it connection} on the edge $e\in \mathcal{E}$.
The triple $(G,\alpha,\nabla)$ that satisfies these conditions is called a {\it GKM graph}, or an {\it $(m,n)$-type GKM graph} if we emphasize the valency of $\Gamma$ and the dimension of the target space of $\alpha$.

We next define the leg bundle over $\Gamma=(G,\alpha,\nabla)$.
\begin{definition}[Leg bundle over a GKM graph]
\label{def-leg-bdl}
Let $\Gamma=(G, \alpha, \nabla)$ be an $(m,n)$-type GKM graph.
We call $\xi$ a {\it (rank $r$) leg bundle} over the GKM graph $\Gamma$ if
the following data is given for $[r]_{G}$:
\begin{enumerate}
\item we assign the element $\xi_{p}^{j}\in (\mathfrak{t}_{\mathbb{Z}}^{n})^{*}$ to every leg $(p,j)$, called a {\it weight} on $(p,j)$;
\item there is the permutation $\sigma_{e}:[r]_{i(e)}\to [r]_{t(e)}$ for every edge $e\in \mathcal{E}$ that satisfies the following congruence relation:
\begin{align*}
\xi_{t(e)}^{\sigma_{e}(j)}-\xi_{i(e)}^{j}\equiv 0 \mod \alpha(e).
\end{align*}
\end{enumerate}
\end{definition}

We also call the collection $\sigma_{\xi}:=\{\sigma_{e}\ |\ e\in \mathcal{E}\}$ a {\it connection} on $\xi$.
A rank $1$ leg bundle over $\Gamma$ is called a {\it line bundle} over $\Gamma$.
For a line bundle $\xi$ over $\Gamma$, the connection $\sigma_{\xi}$ is uniquely determined.
By forgetting legs and their weights, we can define the projection $\pi:\xi\to \Gamma$, see Figure~\ref{figure_triangle}.

\subsection{Leg bundle induced from the vector bundle over a GKM manifold}
\label{sect:2.3}

In this section, we show how to obtain the leg bundle from the equivariant complex vector bundle over a GKM manifold.
For the definition of a GKM manifold see Section \ref{sect:6}; cf. \cite{GKM,GZ}.

Let $\pi\colon \xi\to M$ be a $T^{n}$-equivariant complex rank $r$ vector bundle with effective $T^n$-action over a GKM manifold $M$.
Recall that the collection of zero- and one-dimensional $T$-orbits of a GKM manifold $M$ form a graph, see e.g.~\cite{GZ, Ku16}.
Since $\xi$ is an equivariant vector bundle and $M$ has a non-empty fixed point set $M^{T}$,
the restriction $\xi_{p}$ of $\xi$ to any fiber over the $T$-fixed point $p\in M^{T}$ may be regarded as a $T$-representation.
This $T$-representation decomposes into the irreducible one-dimensional representations:
\begin{align}
\label{decomposition}
\xi_{p}\simeq V(\xi_{p}^{1})\oplus \cdots \oplus V(\xi_{p}^{r}),
\end{align}
where $V(\xi_{p}^{j})$ is the complex one-dimensional $T$-representation space $\xi_{p}^{j}\in \Hm(T, S^{1})\simeq \algt_{\mathbb{Z}}^{*}$ for $j=1,\ldots, r$.
Notice that the subspace of zero and one-dimensional orbits in the $T$-orbit space of $\xi_{p}$ might not be a graph, because in general $\xi_{p}^{1},\ldots, \xi_{p}^{r}$ in $\algt^{*}_{\mathbb{Z}}\simeq \mathbb{Z}^{n}$ are not pairwise linearly independent.
However, the orbit space of each factor in \eqref{decomposition} is  $V(\xi_{p}^{j})/T^{n}\cong \Ro_{+}=\lb x\in \Ro |\ x\geq 0\rb$ (homeomorphic to a half-line).
This leads us to define the non-compact edge (i.e., leg) with the label $\xi_{p}^{j}$ over the GKM graph of $M$.

\begin{example}
Figure~\ref{figure_triangle} illustrates the leg bundle defined by the $T^{2}$-action on the tangent bundle $T\mathbb{C}P^{2}$ over $\mathbb{C}P^{2}$ with the standard $T^{2}$-action.
\begin{figure}[H]
\centering
\begin{tikzpicture}
\begin{scope}[xscale=1, yscale=1]
\fill (0,1) circle (1pt);
\fill (1,-1) circle (1pt);
\fill (-1,-1) circle (1pt);

\draw (0,1)--(1,-1);
\draw (0,1)--(0.5,2);
\node[right] at (0.5, 2) {$\xi_{w}^{2}=-x_{2}$};
\draw (0,1)--(-0.5,2);
\node[left] at (-0.5, 2) {$\xi_{w}^{1}=x_{1}-x_{2}$};

\draw (-1,-1)--(0,1);

\draw (1,-1)--(2,-1);
\node[above] at (2,-1) {$\xi_{v}^{1}=-x_{1}$};
\draw (1,-1)--(1.5,-2);
\node[right] at (1.5, -2) {$\xi_{v}^{2}=x_{2}-x_{1}$};

\draw (1,-1)--(-1,-1);
\draw (-1,-1)--(-2,-1);
\node[above] at (-2, -1) {$\xi_{u}^{1}=x_{1}$};
\draw (-1,-1)--(-1.5,-2);
\node[right] at (-1.5, -2) {$\xi_{u}^{2}=x_{2}$};

\draw[->] (2.5, 0)--(3, 0);
\node[above] at (2.75, 0) {$\pi$};

\fill (5,1) circle (1pt);
\fill (6,-1) circle (1pt);
\fill (4,-1) circle (1pt);

\node[left] at (4,-1) {$u$};
\node[below] at (4.5, -1.1)
{$x_{1}$};
\node[above] at (4, -0.3)
{$x_{2}$};

\node[above] at (5,1) {$w$};
\node[right] at (5.2, 0.8) {$x_{1}-x_{2}$};

\node[right] at (6,-1) {$v$};

\draw (4,-1)--(6,-1);
\draw[->] (4,-1)--(5,-1);
\node[above] at (5, -1) {$e_{1}$};

\draw (4,-1)--(5,1);
\draw[->] (4,-1)--(4.5,0);
\node[right] at (4.5, 0) {$e_{2}$};

\draw (5,1)--(6,-1);
\draw[->] (5,1)--(5.5,0);
\node[right] at (5.5, 0) {$e_{3}$};

\end{scope}
\end{tikzpicture}

\caption{The right graph $\Gamma=(G,\alpha,\nabla)$ is the GKM graph satisfying $\alpha(\overline{e})=-\alpha(e)$, where $\alpha(e_{1})=x_{1}, \alpha(e_{2})$ and $\alpha(e_{3})=x_{1}-x_{2}$.
The left labeled graph $\xi$ is the rank $2$ leg bundle over $\Gamma$ (see the left leg bundle in Figure~\ref{abstract_leg_bdl}), where
the connection $\sigma_{\xi}$ is uniquely determined.
It is well-known that the right GKM graph corresponds to the standard $T^{2}$-action on $\mathbb{C}P^{2}$. By computing the tangential representations of the tangent space $T\mathbb{C}P^{2}$ over the fixed points, we obtain the left leg bundle with labels $\xi_{p}^{k}$ for $p=u,v,w$, where $k=1,2$.}
\label{figure_triangle}
\end{figure}
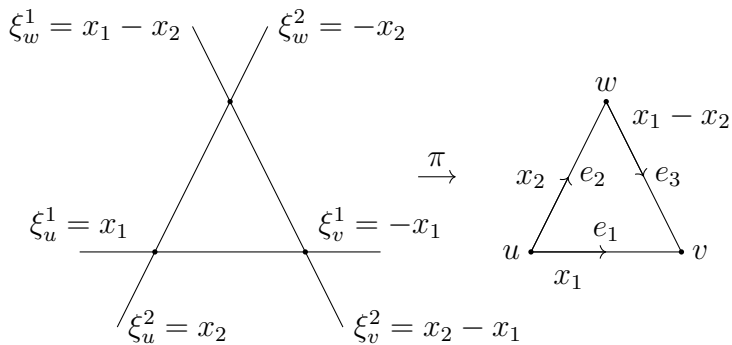
\end{example}

\section{Projectivization of a leg bundle}
\label{sect:3}

Let $\Gamma=(G,\alpha,\nabla)$ be an $(m,n)$-type GKM graph and $\xi$ be any rank $(r+1)$ leg bundle over $\Gamma$.
In this section, we introduce the projectivization $\Pi(\xi)=(P(\xi),\alpha^{P(\xi)},\nabla^{P(\xi)})$ of $\xi$.

\subsection{Vertices and edges}
\label{sect:3.1}
We first introduce the underlying graph of the projectivization $\Pi(\xi)$, say $P(\xi):=(\mathcal{V}^{P(\xi)},\mathcal{E}^{P(\xi)})$.

The set of vertices $\mathcal{V}^{P(\xi)}$ is defined by the set of legs on $[r+1]_{G}$, i.e., set-theoretically,
\begin{align*}
\mathcal{V}^{P(\xi)}:=\bigcup_{p\in \mathcal{V}}[r+1]_{p}=\{(p,l)\ |\ l\in [r+1],\ p\in \mathcal{V}\}.
\end{align*}
The set of edges $\mathcal{E}^{P(\xi)}$ is defined by the set of the following two types of edges:
\begin{description}
\item[vertical]
a {\it vertical edge} $(p,jk)$ connecting two vertices $(p,j),\ (p,k)\in [r+1]_{p}$ if $j\not=k$, where $p$ runs over $\mathcal{V}$ and $j,k$ run over $[r+1]_{p}$ with $j\not=k$;
\item[horizontal]
a {\it horizontal edge} $(e,l)$ for $e\in \mathcal{E}$ and $l\in [r+1]_{i(e)}$ connecting $(i(e),l)$ and $(t(e),\sigma_{e}(l))$.
\end{description}
From this definition, set-theoretically,
\begin{align*}
\mathcal{E}^{P(\xi)}=\left(\bigcup_{p\in \mathcal{V}}\{(p,jk)\ |\ j,k\in [r+1], i\not=j\}\right)\cup
\left(\bigcup_{e\in \mathcal{E}}\{(e,l)\ |\ l\in [r+1]\} \right)
\end{align*}

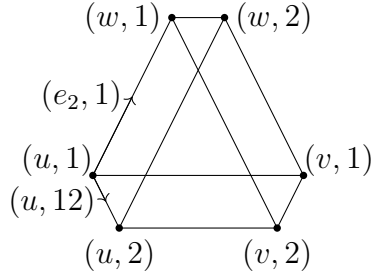
\begin{figure}[H]
\centering
\begin{tikzpicture}
\begin{scope}[xscale=0.7, yscale=0.7]

\fill (-5,-1) coordinate (u1) circle (2pt);
\node[left] at (-4.8,-0.7) {$(u,1)$};
\fill (-4.5,-2) coordinate (u2) circle (2pt);
\node[below] at (-4.5,-2) {$(u,2)$};

\fill (-1,-1) coordinate (v1) circle (2pt);
\node[right] at (-1.2,-0.7) {$(v,1)$};
\fill (-1.5,-2) coordinate (v2) circle (2pt);
\node[below] at (-1.5,-2) {$(v,2)$};

\fill (-3.5,2) coordinate (w1) circle (2pt);
\node[left] at (-3.5,2) {$(w,1)$};
\fill (-2.5,2) coordinate (w2) circle (2pt);
\node[right] at (-2.5,2) {$(w,2)$};

\draw[->] (u1)--(-4.75,-1.5);
\node[left] at (-4.75,-1.5) {$(u,12)$};

\draw[->] (u1)--(-4.25,0.5);
\node[left] at (-4.25,0.5) {$(e_{2},1)$};

\draw (u1)--(u2);
\draw (v1)--(v2);
\draw (w1)--(w2);

\draw (u1)--(v1);
\draw (u2)--(v2);
\draw (v2)--(w1);
\draw (v1)--(w2);
\draw (w1)--(u1);
\draw (w2)--(u2);

\end{scope}
\end{tikzpicture}
\caption{The projectivization $P(\xi)$ of the leg bundle $\xi$ in Figure~\ref{figure_triangle}. Here, $(u,12)$ is the vertical edge connecting $(u,1)$ and $(u,2)$ and $(e_{2},1)$ is the horizontal edge connecting $(u,1)$ and $(w,1)$.}
\label{abstract_projectivization}
\end{figure}

Note that
the reversed orientation edge of the vertical edge $(p,jk)$ is $\overline{(p,jk)}=(p,kj)$
and that of the horizontal edge $(e,l)$ is
$\overline{(e,l)}=(\overline{e},\sigma_{e}(l))$ (also see Definition~\ref{def-leg-bdl}).

\subsection{Label of the projectivization}
\label{sect:3.2}
The {\it label} $\alpha^{P(\xi)}:\mathcal{E}^{P(\xi)}\to (\mathfrak{t}_{\mathbb{Z}}^{n})^{*}$ of the projectivization $\Pi(\xi)$ is defined as follows:
\begin{itemize}
\item $\alpha^{P(\xi)}(p,jk):=\xi_{p}^{k}-\xi_{p}^{j}$, for any vertical edge $(p,jk)\in \mathcal{E}^{P(\xi)}$;
\item $\alpha^{P(\xi)}(e,l):=\alpha(e)$, for any horizontal edge $(e,l)\in \mathcal{E}^{P(\xi)}$.
\end{itemize}

\begin{example}
In Figure~\ref{abstract_projectivization} (also see Figure~\ref{figure_triangle}), for the vertical edge $(u,12)$ and the horizontal edge $(e_{2},1)$, we have
\begin{align*}
& \alpha^{P(\xi)}(u,12)=\xi_{u}^{2}-\xi_{u}^{1}=x_{2}-x_{1};\\
& \alpha^{P(\xi)}(e_{2},1)=\alpha(e_{2})=x_{2}.
\end{align*}
\end{example}

\subsection{The standard connection of the projectivization}
\label{sect:3.3}

In this section, we introduce the {\it standard connection} of the projectivization of $\xi$, denoted by
$\nabla^{P(\xi)}:=\{\nabla_{\epsilon}^{P(\xi)}\ |\ \epsilon\in \mathcal{E}^{P(\xi)}\}$
for $P(\xi)$, and show that it gives the connection on $(P(\xi),\alpha^{P(\xi)})$ in  Theorem~\ref{cong_rel_for_canonical_connection}.
The connection $\nabla^{P(\xi)}$ is defined by the set of the bijective maps
\begin{align*}
\nabla_{\epsilon}^{P(\xi)}:\str_{P(\xi)}(i(\epsilon))\to \str_{P(\xi)}(t(\epsilon)).
\end{align*}
such that
\begin{itemize}
\item $\nabla_{(u,jk)}^{P(\xi)}(u,jl)=(u,kl)$ for every distinct elements $j, k, l\in [r+1]$;
\item $\nabla_{(u,jk)}^{P(\xi)}(e,j)=(e,k)$, where $i(e)=u\in \mathcal{V}$;
\item $\nabla_{(e,l)}^{P(\xi)}(u,lk)=(v,\sigma_{e}(l)\sigma_{e}(k))$, where $i(e)=u, t(e)=v\in \mathcal{V}$ for every distinct elements $l, k\in [r+1]$;
\item $\nabla_{(e,l)}^{P(\xi)}(e',l)=(\nabla_{e}(e'),\sigma_{e}(l))$, where $i(e)=i(e')\in \mathcal{V}$,
\end{itemize}
where we omit $\nabla_{\epsilon}^{P(\xi)}(\epsilon)=\overline{\epsilon}$.

We have the following theorem which is straightforward to prove.
\begin{theorem}
\label{cong_rel_for_canonical_connection}
The collection $\nabla^{P(\xi)}:=\{\nabla_{\epsilon}^{P(\xi)}\ |\ \epsilon\in \mathcal{E}^{P(\xi)}\}$ satisfies the conditions of the connection on $(P(\xi),\alpha^{P(\xi)})$.
\end{theorem}

We call $\nabla^{P(\xi)}$ the {\it standard connection} on $(P(\xi),\alpha^{P(\xi)})$.

\begin{remark}
If $\alpha^{P(\xi)}$ is $2$-independent (see (2) in the conditions of the axial function in Section~\ref{sect:2.2}), then $\Pi(\xi)$ is a GKM graph, and $\Pi(\xi)\to \Gamma$ is a GKM fiber bundle in the sense of \cite{GSZ} (see Definition~\ref{defn:gsz1} in
Section~\ref{sect:6}).
\end{remark}

\subsection{From geometry to combinatorics}
\label{sect:3.4}
Let $\pi:\P(\xi)\to M$ be the projectivization of a $T$-equivariant complex rank $r+1$ vector bundle $\xi\to M$ over a GKM manifold $M$ with a $T$-action, where $T=(S^1)^n$ (see e.g.~\cite[Section 3.1]{Ku10}).
One has the isomorphism of complex vector bundles
\[
T\P(\mathcal{F})\oplus\mathbb{C}\cong (\pi^{*}\mathcal{F}\otimes \gamma)\oplus \pi^{*} TM,
\]
where $\gamma\to\P(\xi)$ denotes the tautological line bundle for $\pi$, and $\mathbb{C}$ denotes the trivial line bundle.
This isomorphism is equivariant with respect to the $T$-action that is trivial on the line summand $\mathbb{C}$.
Therefore, the labels on the horizontal edges are given by those on the respective edges of the GKM-graph for $M$, and it remains to determine the labels of the vertical edges.
Let $p\in M^{T}$.
In this case, the decomposition \eqref{decomposition} becomes the following irreducible decomposition:
\begin{align*}
\xi_p\simeq V(\xi_{p}^{1})\oplus \cdots \oplus V(\xi_{p}^{r+1}).
\end{align*}
The $T$-action on $\xi$ induces the $T$-action on the projectivization $\P(\xi)$.
Therefore, by restricting this action on $p\in M^{T}$, one can define the $T$-action on the projectivization
$\P_{p}(\xi)=\pi^{-1}(p)\cong \P(\xi_p)\cong \mathbb{C}P^{r}$ of the fiber $\xi_p$.
If $\{\xi_{p}^{j}-\xi_{p}^{k}\ |\ j,k\in [r+1]\}$ satisfies the $2$-independence condition for every $p\in M^{T}$, then $\P(\xi)$ is a GKM manifold.
In such case, it is easy to check that the GKM graph of $\P(\xi)$ is $\Pi(\xi)=(P(\xi),\alpha^{P(\xi)}, \nabla^{P(\xi)})$.

\section{Graph equivariant cohomology algebra and its elements}
\label{sect:4}

Let $\Gamma=(G,\alpha,\nabla)$ be a GKM graph.
By definition, the {\it graph equivariant cohomology} (see \cite{GZ}) is the graded $H^{*}(BT)$-subalgebra
\begin{align*}
H^{*}(\Gamma):=\{f:\mathcal{V}\to H^{*}(BT^{n})\ |\ f(i(e))-f(t(e))\equiv 0\mod \alpha(e)\},
\end{align*}
of $\oplus_{u\in \mathcal{V}}H^{*}(BT^{n})$.
It is well known that $H^{*}(\Gamma)$ is isomorphic to $H_{T}^{*}(M)$ of the GKM manifold $M$ with $T$-action under some conditions (see e.g. \cite{GKM,GZ}, \cite[Theorem 2.12]{DKS}).

In this section, we consider some elements in $H^{*}(\Gamma)$ being motivated by the pull-back of the equivariant Chern classes to the fixed points of the torus action on a manifold.
Notice that a similar notion for the GKM graphs (which satisfy $\alpha(e)=-\alpha(\overline{e})$ for every edge $e$) has already been discussed in some papers (e.g., in \cite{GKZ20, Y} about the combinatorial counterpart of the equivariant Chern classes of the invariant almost complex tangent bundle. Also see \cite{P08} for the toric manifolds).

\subsection{Chern classes}
\label{sect:4.1}
Let $\xi$ be a rank $r+1$ leg bundle over $\Gamma$.
For $0\le s\le r+1$,
the {\it $s$-th (equivariant) Chern class} of $\xi$ is the map
\begin{align*}
c_{s}^{T}(\xi)\colon \mathcal{V}\to H^{2s}(BT^{n})
\end{align*}
defined as follows:
\begin{align*}
c_{s}^{T}(\xi)(u):=\mathfrak{S}_{s}(\xi_{u}^{1},\xi_{u}^{2},\ldots, \xi_{u}^{r+1}),
\end{align*}
where
\begin{align*}
\mathfrak{S}_{s}(x_{1},x_{2},\ldots, x_{r+1}):=\sum_{\substack{a_{1}+a_{2}+\cdots+a_{r+1}=s, \\ 0\le a_{j}\le 1}}x_{1}^{a_{1}}\cdot x_{2}^{a_{2}}\cdots x_{r+1}^{a_{r+1}}
\end{align*}
is the elementary symmetric function of degree $s$.
The $s$-th Chern class is an element of the graph equivariant cohomology.
Namely, we have the following lemma.
\begin{lemma}
\label{Chern_class}
One has
$c_{s}^{T}(\xi)\in H^{2s}(\Gamma)$.
\end{lemma}
\begin{proof}
We shall check the congruence relation holds for all edges $e\in \mathcal{E}$.
Let $i(e)=p$, $t(e)=q\in \mathcal{V}$.
By the definition of the leg bundle, there exists an integer $d_{j}$ for every $j=1,\ldots, r+1$ such that
\begin{align*}
\xi_{q}^{\sigma_{e}(j)}=\xi_{p}^{j}+d_{j}\alpha(e),
\end{align*}
where $\sigma_{e}:[r+1]\to [r+1]$ is the connection on $e$.
Thus, by the definition of the $k$-th Chern class, we have that
\begin{align*}
c_{s}^{T}(\xi)(p)-c_{s}^{T}(\xi)(q)=&\mathfrak{S}_{s}(\xi_{p}^{1},\ldots, \xi_{p}^{r+1})-\mathfrak{S}_{s}(\xi_{q}^{1},\ldots, \xi_{q}^{r+1}) \\
=&\mathfrak{S}_{s}(\xi_{p}^{1},\ldots, \xi_{p}^{r+1})-\mathfrak{S}_{s}(\xi_{p}^{1}+d_{1}\alpha(e),\ldots, \xi_{p}^{r+1}+d_{r+1}\alpha(e)) \\
\equiv & \mathfrak{S}_{s}(\xi_{p}^{1},\ldots, \xi_{p}^{r+1})-\mathfrak{S}_{s}(\xi_{p}^{1},\ldots, \xi_{p}^{r+1}) \mod \alpha(e) \\
= & 0.
\end{align*}
This establishes the statement.
\end{proof}

\begin{remark}
\label{classification_line_bundles}
It follows easily from Definition~\ref{def-leg-bdl} that
there is a one-to-one correspondence between the set of
line bundles over the GKM graph $\Gamma$ and $H^{2}(\Gamma)$ by taking the first Chern class of the line bundle and conversely assigning the value of the function in $H^{2}(\Gamma)$ on $p\in \mathcal{V}$ to the unique leg on $p$.
\end{remark}

\subsection{The tautological class}
\label{sect:4.2}
Let $\Pi(\xi)=(P(\xi),\alpha^{P(\xi)}, \nabla^{P(\xi)})$ be the projectivizaiton of $\xi$.
From now on, we suppose that $\Pi(\xi)$ is a GKM graph.

Define the following function:
\begin{align*}
c_{\xi}:\mathcal{V}^{P(\xi)}\to H^{2}(BT^{n}),\quad c_{\xi}(u,l):=\xi_{u}^{l}.
\end{align*}
Then, we have the following lemma (the proof is straightforward).
\begin{lemma}
\label{tautological_class}
One has
$c_{\xi}\in H^{2}(\Pi(\xi))$.
\end{lemma}
We call this class $c_{\xi}$ the {\it tautological class} of $\Pi(\xi)$.
We call the line bundle which corresponds to $c_{\xi}\in H^{2}(\Pi(\xi))$ the {\it tautological line bundle} of $\xi$ over $\Pi(\xi)$ (see Remark~\ref{classification_line_bundles}).

\begin{remark}
Geometrically, $c_{\xi}$ corresponds to the first Chern class of the tautological line bundle of $\P(\xi)$ (see \cite{GH}).
\end{remark}

\section{Combinatorial Borel-Hirzebruch formula and Leray-Hirsch theorem}
\label{sect:5}

In this section, we prove the main theorem of this paper, see Theorem~\ref{main_theorem2}.

\subsection{The injective homomorphism $\varphi$}
\label{sect:5.1}

We first define the homomorphism $\varphi:H^{*}(\Gamma)\to H^{*}(\Pi(\xi))$.
For an element $f\colon \mathcal{V}\to H^{*}(BT^{n})\in H^{*}(\Gamma)$,
the map $\varphi(f)\colon \mathcal{V}^{P(\xi)}\to H^{*}(BT^{n})\in H^{*}(\Pi(\xi))$ is defined by
\begin{align}
\label{embedding}
\varphi(f)(u,l):=f(u).
\end{align}
We have the following straightforward lemma.
\begin{lemma}
\label{induced_hom}
The induced map $\varphi\colon H^{*}(\Gamma)\longrightarrow H^{*}(\Pi(\xi))$ is an injective homomorphism.
\end{lemma}
Notice that $H^{*}(\Pi(\xi))$ is an $H^{*}(\Gamma)$-algebra with respect to the homomorphism $\phi$.
By a slight abuse of the notation, we identify $H^{*}(\Gamma)$ with its image in $H^{*}(\Pi(\xi))$ by using Lemma~\ref{induced_hom}.
In particular, we may regard the $s$-th Chern class $c_{s}^{T}(\xi)\in H^{*}(\Gamma)$ of the leg bundle $\xi$ as an element of $H^{*}(\Pi(\xi))$.

\subsection{Main theorem and preparation}
\label{sect:5.2}

Now we may state the main theorem of this paper.
\begin{theorem}
\label{main_theorem2}
Let $\xi$ be a rank $r+1$ leg bundle over a GKM graph $\Gamma$.
Assume that its projectivization $\Pi(\xi)$ is a GKM graph.
Then, there is the following isomorphism of  $H^*(\Gamma)$-algebras:
\begin{align*}
H^{*}(\Gamma)[\kappa]\big/\biggl( \sum_{s=0}^{r+1} (-1)^{s}c_{s}^{T}(\xi) \cdot \kappa^{r+1-s} \biggr)\simeq H^{*}(\Pi(\xi))\quad s.t.\quad \kappa\mapsto c_{\xi}.
\end{align*}
\end{theorem}
The purpose of this section is to prove Theorem~\ref{main_theorem2}.

To do that, we first put
\begin{itemize}
\item $t:=c_{\xi}\in H^*(\Pi(\xi))$,
\item $c_{s}:=c_{s}^{T}(\xi)\in H^*(\Gamma)\subset H^*(\Pi(\xi))$, where $s=0,\dots,r+1$.
\end{itemize}
Consider the following map:
\begin{align}
\label{the_map}
&\mu\colon H^{*}(\Gamma)[\kappa]/( \sum_{s=0}^{r+1} (-1)^{s}c_{s} \kappa^{r+1-s}) \longrightarrow H^{*}(\Pi(\xi)), \nonumber\\
&\mu\left(\sum_{i=0}^{n}f_{i}\kappa^{i}\right):=\sum_{i=0}^{n}f_{i} t^{i},
\end{align}
where $f_{i}\in H^{*}(\Gamma)$.
We first prove the following lemma:
\begin{lemma}
\label{well-defined}
The map $\mu$ is a well-defined homomorphism of $H^*(\Gamma)$-algebras.
\end{lemma}
\begin{proof}
We claim that $\mu(\sum_{s=0}^{r+1} (-1)^{s}c_{s} \kappa^{r+1-s})=0$ holds.
Let $x:=(p,j)\in \mathcal{V}^{P(\xi)}$ for $p\in \mathcal{V}$ and $j\in [r+1]$.
Notice that the identities $t(x)=\xi_{p}^{j}$ and $c_{s}(x)=\mathfrak{S}_{s}(\xi_{p}^{1},\ldots,\xi_{p}^{r+1})$ hold by the definitions above.
Therefore, we may conduct the following computation:
\begin{align*}
\mu\left(\sum_{s=0}^{r+1} (-1)^{s}c_{s} \kappa^{r+1-s}\right)(x)=&
\sum_{s=0}^{r+1} (-1)^{s}c_{s}(x) t(x)^{r+1-s}\\
=&\sum_{s=0}^{r+1}(-1)^{s}\mathfrak{S}_{s}(\xi_{p}^{1},\ldots, \xi_{p}^{r+1})(\xi_{p}^{j})^{r+1-s} \\
=&\prod_{s=1}^{r+1}\left(\xi_{p}^{j}-\xi_{p}^{s}\right) =0\quad ({\rm by}\ j\in [r+1]).
\end{align*}
Thus, $\mu$ is well defined.
It is easy to check that $\mu$ is an $H^{*}(\Gamma)$-algebra homomorphism.
The proof is complete.
\end{proof}

The following lemma is also needed to prove the main theorem:
\begin{lemma}
\label{free-module}
As an $H^{*}(\Gamma)$-module, there is the following isomorphism:
\begin{align*}
H^{*}(\Gamma)[\kappa]\big/\biggl( \sum_{s=0}^{r+1} (-1)^{s}c_{s}^{T}(\xi) \cdot \kappa^{r+1-s} \biggr)\simeq
H^{*}(\Gamma)\oplus H^{*}(\Gamma)\kappa\oplus \cdots \oplus H^{*}(\Gamma)\kappa^{r}.
\end{align*}
\end{lemma}
\begin{proof}
By the relation of the ring structure on the left-hand side,
the element $\kappa^{r+1}$ can be written by the unique linear combination of $1, \kappa,\ldots, \kappa^{r}$ with the coefficient in $H^{*}(\Gamma)$.
This also shows that any elements in the left-hand side can be always written as the following presentation:
\begin{align*}
f_{0}+f_{1}\kappa+\cdots +f_{r}\kappa^{r}
\end{align*}
for some elements $f_{0},\ldots, f_{r}\in H^{*}(\Gamma)$.
Since there is no $\kappa^{r+1}$ term,
it is easy to see that this presentation is unique.
This proves the statement.
\end{proof}

\subsection{The graph equivariant cohomology on the fiber}
\label{sect:5.3}

To prove the main theorem, we will use the result \cite[Theorem 3.5]{GSZ} for the integer coefficient, see Corollary~\ref{Q_k}.
Prior to this, we prove Lemma~\ref{technical_lemm} that verifies the assumptions in order to apply \cite[Theorem 3.5]{GSZ}, also see the Leray-Hirsch theorem (e.g.~\cite[Theorem 4D.1]{H})
in the case of ordinary equivariant cohomology on the manifold.

To state Lemma~\ref{technical_lemm}, we prepare some notations.
Take a vertex $p\in \mathcal{V}$.
Define $P_{p}(\xi)$ as the subgraph $(\mathcal{V}^{P(\xi)}_{p},\mathcal{E}^{P(\xi)}_{p})$ of $P(\xi)$ which consists of
\begin{description}
\item[vertex] the set of vertices $\mathcal{V}_{p}^{P(\xi)}:=[r+1]_{p}$;
\item[edge] the set of edges $\mathcal{E}_{p}^{P(\xi)}:=\{(p,jk)\ |\ j,k\in [r+1]_{p}\}$, i.e., the vertical edges on $p$.
\end{description}
By the assumption that $\Pi(\xi)$ is a GKM graph, it follows from
Theorem~\ref{cong_rel_for_canonical_connection} that by restricting the axial function and the connection on $P_{p}(\xi)$ we can define the GKM subgraph $\Pi_{p}(\xi)$ whose underlying graph is $P_{p}(\xi)$.
We call the GKM graph $\Pi_{p}(\xi)=(P_{p}(\xi),\alpha_{p}^{P(\xi)})$ the {\it fiber} of $\Pi(\xi)$ on $p\in \mathcal{V}$, where $\alpha_{p}^{P(\xi)}:\mathcal{E}_{p}^{P(\xi)}\to (\mathfrak{t}_{\mathbb{Z}}^{n})^{*}$ is the restriction of the axial function $\alpha^{P(\xi)}$ to $\mathcal{E}_{p}^{P(\xi)}$.
Here, we may omit the connection on $\Pi_{p}(\xi)$ because
$P_{p}(\xi)$ is the complete subgraph with $r+1$ vertices and the standard connection on $P_{p}(\xi)$ is induced from the usual connection on the complete graph (i.e., the $1$-skeleton of the $r$-dimensional simplex).
Therefore, the graph equivariant cohomology ring $H^{*}(\Pi_{p}(\xi))$ is well defined.

In the following proof of Lemma~\ref{technical_lemm}, we use the inductive argument for vertices  (see e.g.~\cite[Lemma 4.4]{MMP} or \cite[Lemma 5.6]{KU}).

\begin{lemma}
\label{technical_lemm}
As an $H^{*}(BT)$-module, $H^{*}(\Pi_{p}(\xi))$ is generated by $\{1, t_{p},t_{p}^{2},\ldots, t_{p}^{r}\}$, where  $t_{p}:=t|_{\mathcal{V}_{p}^{P(\xi)}}$ is the well-defined restriction of $t=c_{\xi}\in H^{*}(\Pi(\xi))$ to the subgraph $P_{p}(\xi)$.
Namely, for every element $X\in H^{*}(\Pi_{p}(\xi))$, there exist polynomials $Q_{0}(p),Q_{1}(p),\ldots, Q_{r}(p)\in H^{*}(BT^{n})\subset H^{*}(\Pi_{p}(\xi))$ such that
\begin{align}\label{eq:module_dec}
X=\sum_{s=0}^{r}Q_{s}(p)t_{p}^{s}.
\end{align}
\end{lemma}
\begin{proof}
Take an element $X\in H^{*}(\Pi_{p}(\xi))$.
Recall $\mathcal{V}_{p}^{P(\xi)}=[r+1]_{p}=\{(p,j)\ |\ j\in [r+1]\}$.
Consider the monomorphism $\iota\colon H^{*}(BT^{n})\to H^{*}(\Pi_{p}(\xi))$ defined by the constant functions.
One has  $X(p,1)\in {\rm Im}~\iota\subset H^{*}(\Pi_{p}(\xi))$ because of $X(p,1)\in H^{*}(BT^{n})$.

We first put $X_{1}:=X-X(p,1)$.
Then the element $X_{1}\in H^{*}(\Pi_{p}(\xi))$ satisfies $X_{1}(p,1)=0$.
By definition of the fiber $\Pi_{p}(\xi)$,
for every $(p,j)$ ($j=2,\ldots, r+1$), it follows from the congruence relations that one has
\begin{align*}
X_{1}(p,j)\equiv 0 \mod \xi_{p}^{1}-\xi_{p}^{j}=t_{p}(p,1)-t_{p}(p,j).
\end{align*}
In other words, for every $j=2,\ldots, r+1$
there exists an element $Y_{1}(p,j)\in H^{*}(BT^{n})$ such that
\begin{align*}
X_{1}(p,j)
=Y_{1}(p,j)\bigl(t_{p}(p,1)-t_{p}(p,j)\bigr).
\end{align*}
We next take the following element in $H^{*}(\Pi_{p}(\xi))$:
\begin{align*}
X_{2}:=X_{1}-Y_{1}(p,2)\bigl(t_{p}(p,1)-t_{p}\bigr),
\end{align*}
where we regard $Y_{1}(p,2), t_{p}(p,1)\in {\rm Im}~\iota\subset H^{*}(\Pi_{p}(\xi))$.
This element satisfies the equalities  $X_{2}(p,1)=X_{2}(p,2)=0$.
So, by the congruence relations, we have
\begin{align*}
X_{2}(p,j)\equiv 0 \mod t_{p}(p,l)-t_{p}(p,j),
\end{align*}
for $j=3,\ldots r+1$ and $l=1,2$.
Therefore, there exists $Y_{2}(p,j)\in H^{*}(BT^{n})$ for $j=3,\ldots, r+1$ such that
\begin{align*}
X_{2}(p,j)=Y_{2}(p,j)\prod_{l=1}^{2}\bigl(t_{p}(p,l)-t_{p}(p,j)\bigr).
\end{align*}
Note that $t_{p}(p,1)-t_{p}(p,j)$ and $t_{p}(p,2)-t_{p}(p,j)$ are linearly independent for any  $j=3,\dots,r+1$ because $\Pi(\xi)$ is a GKM graph.
Similarly, if $X_{k-1}\in H^{*}(\Pi_{p}(\xi))$ satisfies $X_{k-1}(p,l)=0$ for $l=1,\ldots, k-1$, then it follows from the congruence relations that there exists  $Y_{k-1}(p,j)\in H^{*}(BT^{n})$ for $j=k,\ldots,r+1$ such that
\begin{align*}
X_{k-1}(p,j)=Y_{k-1}(p,j)\prod_{l=1}^{k-1}\bigl(t_{p}(p,l)-t_{p}(p,j)\bigr).
\end{align*}
Therefore, if we put
\begin{align*}
X_{k}:=X_{k-1}-Y_{k-1}(p,k)\prod_{l=1}^{k-1}\bigl(t_{p}(p,l)-t_{p}\bigr)\in H^{*}(\Pi_{p}(\xi)),
\end{align*}
then one has the equality $X_{k}(p,l)=0$ for $l=1,\ldots, k$.

Put $Z_{k-1}:=Y_{k-1}(p,k)\prod_{l=1}^{k-1}\bigl(t_{p}(p,l)-t_{p}\bigr)$.
Then, $Z_{k-1}$ is an element generated by $\{1,t_{p},\ldots, t_{p}^{k-1}\}$.
Inductively, we can make
\begin{align*}
X_{r+1}=X_{r}-Z_{r}=0\in H^{*}(\Pi_{p}(\xi)).
\end{align*}
By the constructions as above, we have the equalities
\begin{align*}
X_{r+1}=X-(X(p,1)+Z_{1}+\cdots +Z_{r})=0.
\end{align*}
Since $(X(p,1)+Z_{1}+\cdots +Z_{r})$ is an element generated by $\{1,t_{p},\ldots, t_{p}^{r}\}$,
this proves the lemma.
\end{proof}

\subsection{The proof of the main theorem}
\label{sect:5.4}

By using Lemma~\ref{technical_lemm} and some modification of \cite[Theorem~3.5]{GSZ}, we have the following corollary.
\begin{corollary}
\label{Q_k}
The graph equivariant cohomology $H^{*}(\Pi(\xi))$ is a free $H^{*}(\Gamma)$-module generated by $\{1,t,\ldots,t^{r}\}$, i.e.,
\begin{align}
\label{decomposition_of_f}
H^{*}(\Pi(\xi))\simeq \bigoplus_{s=0}^{r}H^{*}(\Gamma)t^{s}.
\end{align}
\end{corollary}
\begin{proof}
Due to Lemma~\ref{technical_lemm}, we see that $1, t, t^{2},\ldots, t^{r}\in H^{2}(\Pi(\xi))$ satisfy the assumption of \cite[Theorem~3.5]{GSZ}, where $t=c_{\xi}$.
The argument in \cite[Theorem~3.5]{GSZ} is given for the real coefficients.
However,
we can also apply the similar argument of \cite[Theorem~3.5]{GSZ} for the integer coefficient.
This establishes the statement.
\end{proof}

Now we may prove that $\mu$ is an isomorphism.
\begin{lemma}
\label{isomorphism}
Suppose that $\Pi(\xi)$ is a GKM graph. Then the homomorphism $\mu$ is an isomorphism.
\end{lemma}
\begin{proof}
By Lemma~\ref{well-defined}, Lemma~\ref{free-module} and Corollary~\ref{Q_k},
the homomorphism $\mu$ induces the $H^{*}(\Gamma)$-module isomorphism between free $H^{*}(\Gamma)$-modules.
Since $\mu$ is an algebra homomorphism,
$\mu$ is an isomorphism. This establishes the statement.
\end{proof}

Consequently, Theorem~\ref{main_theorem2} follows directly from
Lemma~\ref{isomorphism}.

\subsection{An example of the computation and some remarks}
\label{sect:5.5}
In this section, by applying Theorem~\ref{main_theorem2}, we compute some graph equivariant cohomology.
By the projectivization of Figure~\ref{figure_triangle}, we have the following GKM graph:
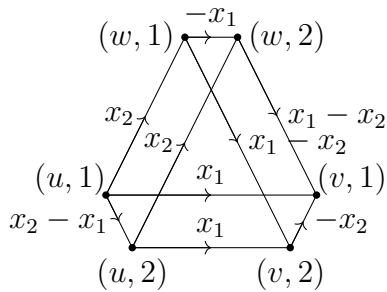
\begin{figure}[H]
\centering
\begin{tikzpicture}
\begin{scope}[xscale=0.7, yscale=0.7]
\fill (-5,-1) coordinate (u1) circle (2pt);
\node[left] at (-4.8,-0.7) {$(u,1)$};
\fill (-4.5,-2) coordinate (u2) circle (2pt);
\node[below] at (-4.5,-2) {$(u,2)$};

\fill (-1,-1) coordinate (v1) circle (2pt);
\node[right] at (-1.2,-0.7) {$(v,1)$};
\fill (-1.5,-2) coordinate (v2) circle (2pt);
\node[below] at (-1.5,-2) {$(v,2)$};

\fill (-3.5,2) coordinate (w1) circle (2pt);
\node[left] at (-3.5,2) {$(w,1)$};
\fill (-2.5,2) coordinate (w2) circle (2pt);
\node[right] at (-2.5,2) {$(w,2)$};

\draw[->] (u1)--(-4.75,-1.5);
\node[left] at (-4.75,-1.5) {$x_{2}-x_{1}$};
\draw[->] (w1)--(-3,2);
\node[above] at (-3,2) {$-x_{1}$};
\draw[->] (v2)--(-1.25,-1.5);
\node[right] at (-1.25,-1.5) {$-x_{2}$};

\draw[->] (u1)--(-4.25,0.5);
\node[left] at (-4.25,0.5) {$x_{2}$};
\draw[->] (u2)--(-3.5,0);
\node[left] at (-3.5,0) {$x_{2}$};
\draw[->] (u1)--(-3,-1);
\node[above] at (-3,-1) {$x_{1}$};
\draw[->] (u2)--(-3,-2);
\node[above] at (-3,-2) {$x_{1}$};
\draw[->] (w1)--(-2.5,0);
\node[right] at (-2.5,0) {$x_{1}-x_{2}$};
\draw[->] (w2)--(-1.75,0.5);
\node[right] at (-1.75,0.5) {$x_{1}-x_{2}$};

\draw (u1)--(u2);
\draw (v1)--(v2);
\draw (w1)--(w2);

\draw (u1)--(v1);
\draw (u2)--(v2);
\draw (v2)--(w1);
\draw (v1)--(w2);
\draw (w1)--(u1);
\draw (w2)--(u2);
\end{scope}
\end{tikzpicture}
\caption{The projectivization $\Pi(\xi)$ of the leg bundle in  Figure~\ref{figure_triangle}.}
\label{flag_mfd}
\end{figure}
By Theorem~\ref{main_theorem2}, we have that
\begin{align*}
H^{*}(\Pi(\xi))\simeq H^{*}(\Gamma)[\kappa]/\left( \kappa^{2}
-c_{1}^{T}(\xi)\cdot \kappa
+c_{2}^{T}(\xi) \right),
\end{align*}
where $\kappa=c_{\xi}\in H^{2}(\Pi(\xi))$.
In this case,
the GKM graph $\Gamma$ (see the right graph in Figure~\ref{figure_triangle}) is a torus graph.
Therefore, we can compute the graph equivaraint cohomology $H^{*}(\Gamma)$ by using \cite{MMP}, and we obtain:
\begin{align*}
H^{*}(\Gamma)\simeq \mathbb{Z}[\tau_{1},\tau_{2},\tau_{3}]/\langle \tau_{1}\tau_{2}\tau_{3} \rangle,
\end{align*}
where $\tau_{i}$ is the equivariant Thom class of the edge $e_{i}$, where $i=1,2,3$, in Figure~\ref{figure_triangle}.
Moreover, it is easy to check that there are the following equalities:
\begin{align*}
& c_{1}^{T}(\xi)=\tau_{1}+\tau_{2}+\tau_{3}=\mathfrak{S}_{1}(\tau_{1},\tau_{2},\tau_{3}); \\
& c_{2}^{T}(\xi)=\tau_{1}\tau_{2}+\tau_{2}\tau_{3}+\tau_{3}\tau_{1}=\mathfrak{S}_{2}(\tau_{1},\tau_{2},\tau_{3}).
\end{align*}
Therefore, we have the following ring structure for $H^{*}(\Pi(\xi))$:
\begin{align}
\label{ring_example}
H^{*}(\Pi(\xi))\simeq \mathbb{Z}[\tau_{1},\tau_{2},\tau_{3},\kappa]/\left(\tau_{1}\tau_{2}\tau_{3},\ \kappa^{2}-(\tau_{1}+\tau_{2}+\tau_{3})\kappa+(\tau_{1}\tau_{2}+\tau_{2}\tau_{3}+\tau_{3}\tau_{1}) \right).
\end{align}

\begin{remark}
There is the following isomorphism of algebras
\begin{align*}
H_{T^{2}}^{*}(\P(T\mathbb{C}P^{2}))\simeq H^{*}(\Pi(\xi)).
\end{align*}
This follows by theorem of \cite{GKM} (\cite{FP} over $\Z$-coefficients), because $H^{odd}(\P(T\mathbb{C} P^2))=0$.
(Namely, the last equality implies equivariant formality of the $T$-action on $\P(T\mathbb{C} P^2)$, i.e. this is a GKM manifold.)
\end{remark}

\begin{remark}
One can also prove that $\P(T\mathbb{C}P^{2})$ is $T^{2}$-equivariantly diffeomorphic to the flag manifold $\mathcal{F}l(\mathbb{C}^{3})$.
The equivariant cohomology of the $T^{2}$-action on $SU(3)/T^{2}\cong \mathcal{F}l(\mathbb{C}^{3})$ can also be computed by the well-known Borel description (see \cite{FIM} for the GKM theoretical point of view).
Notice that the $T^{2}$-aciton on $SU(3)/T^{2}$ is non-effective.
On the other hand, the ring structure given by \eqref{ring_example} corresponds to
the equivariant cohomology of the effective $T^{2}$-action on $\mathcal{F}l(\mathbb{C}^{3})$ (also see \cite[Remark 4.5]{KKLS} and the computation in \cite{KLSS}).
\end{remark}

\begin{remark}
In the case when $\Pi(\xi)$ is the GKM graph of a projectivization $\mathbb{P}(\xi)$ of a torus-equivariant complex vector bundle over a GKM manifold,
Corollary~\ref{Q_k} is nothing but the {\it Leray-Hirsch theorem}, see \cite[Theorem 4D.1]{H}.
Even though we do not assume the existence of such a geometric origin,
the module structure of $H^{*}(\Pi(\xi))$ with the $\mathbb{R}$-coefficient was proved in \cite[Theorem 3.5]{GSZ} under some conditions corresponding to Lemma~\ref{technical_lemm} similar to the classical Leray-Hirsch theorem.
Our result generalizes \cite[Theorem 3.5]{GSZ} to the $\mathbb{Z}$-coefficient by proving the existence of module generators in Lemma~\ref{technical_lemm}. 
This describes the $H^{*}(BT)$-module structure of the graph equivariant cohomology.

Furthermore, we prove Theorem~\ref{main_theorem2}, i.e., a version of the {\it Borel-Hirzebruch formula} (see \cite[15.1 (3)]{BH}) for the class of $2$-independent projectivization GKM-graphs (not necessarily realized by topological fiber bundles) with the $\mathbb{Z}$-coefficients.
If there is the geometric origin for $\Pi(\xi)$, then we obtain it by using the  {\it Borel-Hirzebruch formula} proved in \cite[15.1 (3)]{BH}.
Theorem~\ref{main_theorem2} describes the $H^{*}(BT)$-algebra structure with the $\mathbb{Z}$-coefficient of the graph equivariant cohomology.
\end{remark}

\section{Realization of the projective GKM fiber bundles by projectivizations of leg bundles}
\label{sect:6}

In this final section, we study the realization of the projective bundle from the complex vector bundle.

\subsection{Geometric realization}
\label{sect:6.1}

We first recall the non-equivariant case.
Consider a $\C P^r$-bundle $p\colon P\to M$ over a manifold $M$ with the structure group $PGL_{r+1}(\mathbb{C})=GL_{r+1}(\mathbb{C})/\mathbb{C}^{*}$, where $\mathbb{C}^{*}$ denotes the diagonal matrix.
The exact sequence of groups
\begin{align*}
\mathbb{C}^{*}\to GL_{r+1}(\mathbb{C})\to PGL_{r+1}(\mathbb{C})
\end{align*}
induces a long exact sequence of sheaf cohomologies
\begin{align*}
H^{1}(M;GL_{r+1}(\mathbb{C}))\to H^{1}(M;PGL_{r+1}(\mathbb{C}))\to H^{2}(M;\mathbb{C}^{*}).
\end{align*}
Using the isomorphism $H^{2}(M;\mathbb{C}^{*})\simeq H^{3}(M;\mathbb{Z})$ derived from the exponential sequence (referred to, e.g.~\cite{GH}),
this sequence reveals that there exists an obstruction in $H^{3}(M;\mathbb{Z})$ to determine whether the bundle $p\colon P\to M$ is induced from the projectivization $\P(\xi)$ of a rank $r+1$ complex vector bundle $\xi$.
If $M$ satisfies $H^{odd}(M;\mathbb{Z})=0$, then $P\cong \P(\xi)$ holds.
A GKM manifold $M$ is called an {\it equivariantly formal} (over $\mathbb{Z}$-coefficient in this paper, see \cite{GKM,FP}) if $H^{odd}(M;\mathbb{Z})=0$.
Therefore, we have the following proposition.
\begin{lemma}
\label{non-equivariant}
Let $M$ be an equivariantly formal GKM manifold.
Then, for any $\C P^r$-bundle $p\colon P\to M$ over $M$, there exists a rank $(r+1)$-complex vector bundle $\xi$ such that $\P(\xi)$ is isomorphic to $P$ as a $\mathbb{C}P^{r}$-bundle.
\end{lemma}

Using the result in \cite[Proposition 6.2]{Su} (also see \cite[Corollary 1.4]{HY}), we have the equivariant version of Lemma~\ref{non-equivariant} as follows:
\begin{theorem}
\label{equivariant-bundle}
Let $M$ be an equivariantly formal GKM manifold with $T^{n}$-action.
Then, for any $T^{n}$-equivariant $\C P^r$-bundle $p\colon P\to M$ over $M$, there exists a rank $(r+1)$-complex $T^{n}$-equivariant vector bundle $\xi$ such that $\P(\xi)$ is isomorphic to $P$ as a $T^{n}$-equivariant $\mathbb{C}P^{r}$-bundle.
\end{theorem}
\begin{proof}
By $H^{odd}(M;\mathbb{Z})=0$ and Lemma~\ref{non-equivariant}, there exists a rank $(r+1)$-complex vector bundle $\xi$ such that $\P(\xi)\cong P$ as a $\mathbb{C}P^{r}$-bundle.
We claim that that any such $\xi$ admits a structure of a $T^n$-equivariant vector bundle so that the equivariant bundle $\P(\xi)$ is isomorphic to $P$.

Since $M$ is compact,
by fixing a Hermitian metric on the vector bundle $\xi$, we may take the $(2r+1)$-dimensional unit sphere bundle in $\xi\to M$, say $\mathbb{S}(\xi)\to M$.
Note that the structure group of the sphere bundle $\mathbb{S}(\xi)\to M$ is isomoprhic to the unitary group $U(r+1)$.
Consider the diagonal circle subgroup $S^1$ of the structure group $U(r+1)$ for $\mathbb{S}(\xi)$.
It follows directly from the definition of the projectivization that the $S^1$-action on the fibers of $\mathbb{S}(\xi)$ is free.
This implies that $\mathbb{S}(\xi)/S^{1}\cong \P(\xi)\cong P$ hold.
Therefore, $\mathbb{S}(\xi)$ can be regarded as a principal $S^{1}$-bundle over $P$.
Since $H^{1}(\mathbb{C}P^{r};\mathbb{Z})=H^{1}(M;\mathbb{Z})=0$, we have $H^{1}(P;\mathbb{Z})=0$.
Hence, $P$ satisfies the condition in \cite[Proposition 6.2]{Su} (also see \cite[Corollary 1.4]{HY}).
Therefore, the principal $S^1$-bundle $\mathbb{S}(\xi)\to P$ admits the lifting of the $T^{n}$-action on $P$, i.e., $\mathbb{S}(\xi)$ has a $T^{n}$-action such that $\mathbb{S}(\xi)\to P$ is $T^{n}$-equivariant.
Clearly, the actions of $S^1$ and $T^n$ on $\mathbb{S}(\xi)$ commute with each other.
Therefore, $T^{n}$ also acts on $\mathbb{P}(\xi)$, and the equivariant bundle $\P(\xi)$ is isomorphic to $P$.
We denote this $T^{n}$-action on $\mathbb{P}(\xi)$ by $\varphi$.

We next claim that $\varphi$ is obtained from the projectivization of the equivariant vector bundle.
By the above argument,
the sphere bundle $\mathbb{S}(\xi)\to M$ is obtained from the composition $\rho:\mathbb{S}(\xi)\to \mathbb{P}(\xi)\to M$ and its structure group is $U(r+1)$.
Therefore, by the definition of lifting, the lifted $T$-action satisfies that the lift of $t:M\ni x\mapsto tx\in  M$ to the fiber $t^{*}:\rho^{-1}(x)=\mathbb{S}(\xi_{x})\to \rho^{-1}(tx)=\mathbb{S}(\xi_{tx})$ is in $U(r+1)$ for all $x\in M$ and $t\in T$, where $\xi_{x}\simeq \mathbb{C}^{r+1}$ is the fiber of $\xi$ over $x\in M$.
Moreover, by the above argument, we have that there is the $T^{n}$-equivariant complex line bundle $\mathbb{S}(\xi)\times_{S^{1}}\mathbb{C}\to \mathbb{P}(\xi)$, where $S^{1}$ acts on $\mathbb{C}$ by the scaler multiplication.
By removing the zero section of this line bundle, we have that $\mathbb{S}(\xi)\times_{S^{1}}\mathbb{C}^{*}\simeq \mathbb{S}(\xi)\times\mathbb{R}_{>0}\cong \xi^{\times}\to M$,
where $\mathbb{C}^{*}:=\mathbb{C}\setminus \{0\}$, $\mathbb{R}_{>0}:=\{x\in \mathbb{R}\ |\ x>0\}$. Here, $\xi^{\times}$ is the $\mathbb{C}_{0}^{r+1}(:=\mathbb{C}^{r+1}\setminus \{0\})$-bundle over $M$ by removing the zero section from $\xi$.
This shows that $\xi^{\times}$ is a $T^{n}$-equivariant $\mathbb{C}_{0}^{r+1}$-bundle over $M$.
Moreover, for each element $t\in T$, we may choose the lifting $t^{*}:\mathbb{S}(\xi_{x})\times \mathbb{R}_{>0}\to \mathbb{S}(\xi_{tx})\times \mathbb{R}_{>0}$ of the diffeomorphism $t:M\ni x\to tx\in M$ as an element in $U(r+1)\times \{id\}\subset GL_{r+1}(\mathbb{C})$.
The obtained $T^{n}$-action on $\xi^{\times}$ extends to the continuous $T^{n}$-action on $\xi$ uniquely by declaring that the zero section of $\xi\to M$ is fiberwise invariant; moreover, the lift of $t:M\ni x\mapsto tx\in  M$, say $t^{*}:\xi_{x}\to \xi_{tx}$, is linear for all $x\in M$ and $t\in T$.
Therefore, the induced $T^{n}$-action on $\mathbb{P}(\xi)=\xi^{\times}/\mathbb{C}^{*}$ equals the original $T^{n}$-action $\varphi$ on $\mathbb{P}(\xi)$ which is the projectivization of the equivariant $T^{n}$-bundle $\xi$.
This establishes the statement.
\end{proof}

\begin{remark}
\label{equivariant-bundle-remark}
Note that in Theorem~\ref{equivariant-bundle}, $P$ may not be a GKM manifold.
\end{remark}

\subsection{Combinatorial projective bundle}
\label{sect:6.2}
In this section, we study the GKM graph theoretical analogue of Theorem~\ref{equivariant-bundle}.
To do that, we shall define
the notion of a {\it projective GKM fiber bundle} which is a GKM fiber bundle in the sense of \cite{GSZ} whose fiber is the complete graph $K_{r+1}$ with the standard connection.

We first recall some of the definitions from \cite{GSZ}.
Let $P=(\mathcal{V}^{P},\mathcal{E}^{P})$ and $G=(\mathcal{V}^{G},\mathcal{E}^{G})$ be connected graphs.
A {\it graph morphism} $\pi\colon P\to G$ is defined by a set-theoretical map $\pi:\mathcal{V}^{P}\sqcup \mathcal{E}^{P}\to \mathcal{V}^{G}\sqcup \mathcal{E}^{G}$ such that
\begin{itemize}
\item for every $u\in \mathcal{V}^{P}$, $\pi(u)\in \mathcal{V}^{G}$;
\item for every $e\in \mathcal{E}^{P}$, either $\pi(i(e))=\pi(t(e))=\pi(e)\in \mathcal{V}^{G}$ or $\pi(e)\in \mathcal{E}^{G}$ with $\pi(i(e))=i(\pi(e))$.
\end{itemize}
The edge $e\in \mathcal{E}^{P}$ is called a {\it vertical} edge if $\pi(e)\in \mathcal{V}^{G}$ holds.
Otherwise $e\in \mathcal{E}^{P}$ is called a {\it horizontal} edge (Cf. Section~\ref{sect:3.1}).
For a vertex $p\in \mathcal{V}^{P}$,
let $\mathcal{H}_{p}\subset \mathcal{E}^{P}$ be the set of all horizontal edges with the initial vertex $p$, and
let $\mathcal{E}^{\perp}_{p}$ be the set of vertical edges with initial vertex $p$.
A graph morphism $\pi\colon P\to G$ is called a {\it fibration of graphs} ({\it graph fibration} for short) if the restriction map $\mathcal{H}_{p}\to star_{G}(\pi(p))$ is a bijection for every $p\in \mathcal{V}_{P}$.

\begin{definition}[\cite{GSZ}]
\label{defn:gsz1}
Let $\Pi=(P,\alpha^{P},\nabla^{P})$ and $\Gamma=(G,\alpha,\nabla)$ be GKM graphs.
A morphism $\pi\colon \Pi\to \Gamma$ of GKM-graphs is called a \textit{GKM fibration} if
\begin{enumerate}[label=(\roman*)]
\item $\pi:P\to G$ is a fibration of graphs;
\item if $\widetilde{e}\in \mathcal{E}^{P}$ is a lift of $e\in \mathcal{E}^{G}$, i.e., $\pi(\widetilde{e})=e$, then $\alpha^{P}(\widetilde{e})=\alpha(e)$;
\item the connection $\nabla^{P}$ sends vertical and horizontal edges to vertical and horizontal, respectively;
\item one has $\pi(\nabla^{P}|_{\mathcal{H}_{p}})=\nabla|_{star_{G}(\pi(p))}$.
\end{enumerate}
\end{definition}

Let $\pi\colon \Pi\to \Gamma$ be a GKM fibration.
Notice that for every vertex $q\in \mathcal{V}^{G}$ the preimage $\pi^{-1}(q)$ is a subgraph of $P$.
The restrictions of $\alpha^{P}$ and $\nabla^{P}$ of $\Pi$ to $\pi^{-1}(q)$ induce the well-defined GKM subgraph $\Pi_{q}:=(\pi^{-1}(q), \alpha^{P}|_{\pi^{-1}(q)},\nabla^{P}|_{\pi^{-1}(q)})$, where $\alpha^{P}|_{\pi^{-1}(q)}:\mathcal{E}^{\pi^{-1}(q)}\to \mathfrak{t}_{q}$ for
$\mathfrak{t}_{q}:=\mathbb{Z}\langle \alpha^{P}(e)\ |\ e\in \mathcal{E}^{\pi^{-1}(q)} \rangle\subset \mathfrak{t}$.
Here, the symbol $\mathbb{Z}\langle \alpha^{P}(e)\ |\ e\in \mathcal{E}^{\pi^{-1}(q)} \rangle$
represents the linear space over $\mathbb{Z}$ spanned by $\alpha^{P}(e)$'s.
In addition, for any edge $e\in \mathcal{E}^{G}$, define the map
\[
\Phi_{e}\colon \mathcal{V}^{\pi^{-1}(i(e))}\to \mathcal{V}^{\pi^{-1}(t(e))},\ i(\widetilde{e})\mapsto t(\widetilde{e}),
\]
where $\widetilde{e}\in \mathcal{E}^{P}$ is every lift of $e\in \mathcal{E}^{G}$.
If $\Phi_{e}$ induces an isomorphism of graphs for any $e\in \mathcal{E}^{G}$,
then the graph fibration $\pi:P\to G$ is called  a \textit{graph fiber bundle}.

\begin{definition}[\cite{GSZ}]
\label{defn:gsz2}
The GKM fibration $\pi:\Pi\to \Gamma$ is called a \textit{GKM fiber bundle} if the following conditions are satisfied:
\begin{enumerate}[label=(\roman*)]
\item the map $\pi:P\to G$ is a graph fiber bundle;
\item the map $\Phi_{e}$ is compatible with the connection $\nabla^{P}$ of $\Pi$ for every edge $e\in \mathcal{E}^{G}$;
\item the map $\Phi_{e}$ induces the isomorphism of GKM graphs from $\Pi_{i(e)}$ to $\Pi_{t(e)}$ for every $e\in \mathcal{E}^{G}$ up to a linear isomorphism $\Psi_{e}\colon \mathfrak{t}_{i(e)}\to \mathfrak{t}_{t(e)}$.
\end{enumerate}
The GKM graph $\Pi_{p}=(\pi^{-1}(p),\alpha^{P}|_{\pi^{-1}(p)},\nabla^{P}|_{\pi^{-1}(p)})$, is called the fiber of the GKM fiber bundle $\pi$ at $p\in \mathcal{V}^{G}$.
\end{definition}

\begin{remark}
By the definition of a leg bundle, see Definition~\ref{def-leg-bdl}, a leg bundle $\xi\to\Gamma$ may be regarded as a GKM fiber bundle with any fiber consisting of a single vertex $p\in \mathcal{V}^{G}$ with the $r$ (non-compact) edges $[r]_{p}$.
\end{remark}

Now we may define the projective GKM fiber bundle.
\begin{definition}
\label{general_proj_bdl}
Let $\Pi$ and $\Gamma$ be GKM graphs.
A GKM fiber bundle $\pi:\Pi\to \Gamma$ is called \textit{projective} if its fiber at some point is isomorphic to the GKM graph on the complete graph $K_{r+1}$ with vertices $[r+1]$, and the axial function on $\Pi$ satisfies the identity
\begin{equation}\label{eq:projdefinition}
\alpha(p,jk)=\alpha(p,jl)-\alpha(p,kl),\ j,k,l\in[r+1],
\end{equation}
where $\alpha$ is an axial function on $\Pi$ and $(p,jk)\in \mathcal{E}^{\pi^{-1}(p)}$
denotes the edge connecting two vertices $(p,j), (p,k)\in \mathcal{V}^{\pi^{-1}(p)}=[r+1]$.
\end{definition}

\begin{remark}
\label{remark_on_projectivization}
The connection of the fiber $K_{r+1}$ of the projective GKM fiber bundle can be chosen as the standard one, i.e. any $3$-cycle in $K_{r+1}$ is parallel transport-invariant which follows easily by \eqref{eq:projdefinition} and $2$-independence of $\Pi$.
\end{remark}

\subsection{Realization problem}
\label{sect:6.3}

Motivated by Theorem~\ref{equivariant-bundle}, we ask the following question.
\begin{problem}
Let $\pi:\Pi\to \Gamma$ be a projective GKM fiber bundle with fiber $K_{r+1}$, where
$\Pi=(P,\alpha^{P},\nabla^{P})$ and $\Gamma=(G,\alpha,\nabla)$.
Does there exist a rank $r+1$ leg bundle $\xi$ over $\Gamma$ such that $\Pi\to \Gamma$ is equal to the projectivization $\Pi(\xi)\to\Gamma$?
\end{problem}

In this section,
we shall answer to this question provided that certain conditions hold.

\begin{proposition}
\label{realization}
Let $\pi:\Pi\to \Gamma$ be a projective GKM fiber bundle with fiber $K_{r+1}$, where
$\Pi=(P,\alpha^{P},\nabla^{P})$ and $\Gamma=(G,\alpha,\nabla)$.
Assume that $\Pi$ has a $\nabla^{P}$-invariant subgraph $\Pi'$ which is a degree $1$ cover of $\Gamma$ for the projection $\pi$, i.e.,
there is the section $\sigma:\Gamma=(G,\alpha,\nabla)\to \Pi=(P,\alpha^{P},\nabla^{P})$ and we can identify $\Pi'=\Gamma$ as its image.
Then, there exists a rank $r+1$ leg bundle $\xi$ over $\Gamma$ such that $\Pi\to \Gamma$ is equal to the projectivization $\Pi(\xi)\to\Gamma$.
\end{proposition}
\begin{proof}
Let $p\in \mathcal{V}^{G}$. We may put an order of the vertices on $\pi^{-1}(p)\simeq K_{r+1}$ as
\begin{align*}
(p,1),\ldots, (p,r+1),
\end{align*}
and the edges which connecting $(p,i)$ and $(p,j)$ as $(p,ij)$. By definition of the projective GKM fiber bundle, for every edge $e\in \mathcal{E}^{G}$ with $p=i(e)$ and $q=t(e)$, there exists the isomorphism $\Phi_{e}:\Pi_{p}\to \Pi_{q}$.
So there is the bijection between their vertices; we denote this as $(p,i)\mapsto (q,\sigma_{e}(i))$ by $\sigma_{e}:[r+1]_{p}\to [r+1]_{q}$, where $[r+1]_{p}:=\{(p,i)\ |\ i=1,\ldots, r+1\}$.
Since there is a GKM subgraph $\Pi'(=\Gamma)$ in $\Pi$, we may assume that
\begin{align*}
\sigma_{e}(r+1)=r+1
\end{align*}
for all $e\in \mathcal{E}^{G}$.
Moreover, because $\Phi_{e}$ is compatible with $\nabla^{P}$, the equality $\nabla^{P}_{\widetilde{e}}(p,ij)=(q,\sigma_{e}(i)\sigma_{e}(j))$ holds for the given order of vertices $[r+1]_{p}$, where $\widetilde{e}\in \mathcal{E}^{P}$ is a lift of $e\in \mathcal{E}^{G}$ with $i(\widetilde{e})=(p,i)$.
Note that by the congruence relation of $\Pi$, we have the following relations:
\begin{align}
\label{sect6:cong}
\alpha^{P}(q,\sigma_{e}(i)\sigma_{e}(j))-\alpha^{P}(p,ij)\equiv 0\mod \alpha(e).
\end{align}

Now we may define the rank $(r+1)$ leg bundle $\xi$ over $\Gamma$. We first define $\xi$ combinatorially as follows:
\begin{itemize}
\item The vertices and edges are equal to $G$;
\item The legs over $p\in \mathcal{V}^{G}$ are $(p,j)$ for $j=1,\ldots, r+1$;
\item The collection of bijective maps $\sigma_{\xi}:=\{\sigma_{e}:[r+1]_{i(e)}\to [r+1]_{t(e)}\ |\ e\in \mathcal{E}^{G}\}$.
\end{itemize}
Then, it is easy to check that $P(\xi)=P$ as the abstract graph.

In order to define the label on $\xi$,
because the labels on edges of $\xi$ are the same with the axial functions on $\Gamma$, it is enough to define the label $\xi_{p}^{i}$ on each leg $(p,i)$.
For the fixed vertex $p\in \mathcal{V}^{G}$, put $\alpha_{p,j}:=\alpha^{P}(p,r+1j)$, $j=1,\ldots ,r$ and $p\in \mathcal{V}^{G}$.
Define
\begin{align*}
\xi_{p}^{1}:=-\alpha_{p,1},\ldots , \xi_{p}^{r}:=-\alpha_{p,r}, \xi_{p}^{r+1}:=0.
\end{align*}
Then, it is easy to check that the projectivization $\Pi(\xi)$ of $\xi$ coincides with $\Pi$.
Moreover, for $\sigma_{e}:[r+1]_{p}\to [r+1]_{q}$ (where $p:=i(e), q:=t(e)$),
the following relation holds:
\begin{enumerate}
\item if $j=r+1$, then $\sigma_{e}(j)=r+1$ and $\xi_{q}^{r+1}-\xi_{p}^{r+1}=0-0=0 \mod \alpha(e)$;
\item if $j\not=r+1$, then by \eqref{sect6:cong}
\begin{align*}
\xi_{q}^{\sigma_{e}(j)}-\xi_{p}^{j}=&-\alpha_{q,\sigma_{e}(j)}+\alpha_{p,j} \\
=&-\alpha^{P}(q,r+1\sigma_{e}(j))+\alpha^{P}(p,r+1j) \\
=&-\alpha^{P}(q,\sigma_{e}(r+1)\sigma_{e}(j))+\alpha^{P}(p,r+1j) \mod \alpha(e).
\end{align*}
\end{enumerate}
Therefore, $\xi$ is a leg bundle with $\Pi(\xi)=\Pi$.
This establishes the statement.
\end{proof}

\begin{remark}
In general, the construction of $\xi$ in the proof of Proposition~\ref{realization} from $\Pi$ does not work.
For instance, it is easy to check that we cannot define such $\xi$ for the projective bundle in Figure~\ref{flag_mfd}.
\end{remark}

\begin{remark}
For an arbitrary projective bundle $\Pi$ over a GKM graph $\Gamma$,
if one allows the labels on legs of a leg bundle $\xi$ with values in $\mathbb{Q}^k$, then we can construct a leg bundle $\xi$ such that $\Pi(\xi)=\Pi$.
We call $\xi$ a {\it rational leg bundle}.
In brief, the reason for the existence of such $\xi$ is as follows.
For a rational leg bundle $\xi$, the following definition of a leg bundle $\xi$ makes sense:
\[
\xi_{p}^{i}:=\frac{1}{r+1}\sum_{j\neq i} \alpha^{P}(p,ij),
\]
where $\alpha^{P}(p,ij)$ is the axial function of the vertical edge $(p,ij)$ of $\Pi$ over $p\in \mathcal{V}^{G}$.
Then, the following equalities hold:
\begin{align*}
\xi_{u}^{j}-\xi_{u}^{k}
&=\frac{1}{r+1} \sum_{l\neq j}\alpha^{P}(u,jl)-\frac{1}{r+1} \sum_{l\neq k}\alpha^{P}(u,kl) \\
&=\frac{1}{r+1}\bigl(\alpha^{P}(u,jk)+ \sum_{l\neq j,k}\alpha^{P}(u,jl)\bigr)-\frac{1}{r+1} \bigl(\alpha^{P}(u,kj)+\sum_{l\neq j,k}\alpha^{P}(u,kl)\bigr) \\
&=\frac{1}{r+1}\bigl(2\alpha^{P}(u,jk)+\sum_{l\neq j,k}\alpha^{P}(u,jk)\bigr)\quad (\text{because of Definition~\ref{general_proj_bdl}}) \\
&=\alpha^{P}(u,jk).
\end{align*}
This shows that $\Pi=\Pi(\xi)$.
\end{remark}

\begin{figure}[H]
\centering
\begin{tikzpicture}
\begin{scope}[xscale=1, yscale=1]
\fill (0,1) circle (1pt);
\fill (1,-1) circle (1pt);
\fill (-1,-1) circle (1pt);

\draw (0,1)--(1,-1);
\draw (0,1)--(0.5,2);
\node[right] at (0.5, 2) {$\xi_{w}^{2}=-\frac{1}{2}x_{1}$};
\draw (0,1)--(-0.5,2);
\node[left] at (-0.5, 2) {$\xi_{w}^{1}=\frac{1}{2}x_{1}$};

\draw (-1,-1)--(0,1);

\draw (1,-1)--(2,-1);
\node[above] at (1.5,-1) {$\xi_{v}^{1}=-\frac{1}{2}x_{2}$};
\draw (1,-1)--(1.5,-2);
\node[right] at (1.5, -2) {$\xi_{v}^{2}=\frac{1}{2}x_{2}$};

\draw (1,-1)--(-1,-1);
\draw (-1,-1)--(-2,-1);
\node[above] at (-2, -1) {$\xi_{u}^{1}=\frac{1}{2}(x_{1}-x_{2})$};
\draw (-1,-1)--(-1.5,-2);
\node[right] at (-1.5, -2) {$\xi_{u}^{2}=\frac{1}{2}(x_{2}-x_{1})$};




\draw[->] (2.5, 0)--(3, 0);

\fill (5,1) circle (1pt);
\fill (6,-1) circle (1pt);
\fill (4,-1) circle (1pt);

\node[left] at (4,-1) {$u$};
\node[below] at (4.2, -1.1)
{$\alpha(e_{1})=x_{1}$};
\node[above] at (4.2, -0.8)
{$\alpha(e_{2})=x_{2}$};

\node[above] at (5,1) {$w$};
\node[right] at (5.2, 0.8) {$\alpha(e_{3})=x_{1}-x_{2}$};

\node[right] at (6,-1) {$v$};

\draw (4,-1)--(6,-1);
\draw[->] (4,-1)--(5,-1);
\node[below] at (5, -1) {$e_{1}$};

\draw (4,-1)--(5,1);
\draw[->] (4,-1)--(4.5,0);
\node[left] at (4.5, 0) {$e_{2}$};

\draw (5,1)--(6,-1);
\draw[->] (5,1)--(5.5,0);
\node[right] at (5.5, 0) {$e_{3}$};

\end{scope}
\end{tikzpicture}
\caption{The rank $2$ leg bundle $\xi$, where
the projective GKM fiber bundle $\Pi$ in Figure~\ref{flag_mfd} can be obtained by the projectivization $\Pi(\xi)$.}
\label{raional_leg_bundle}
\end{figure}
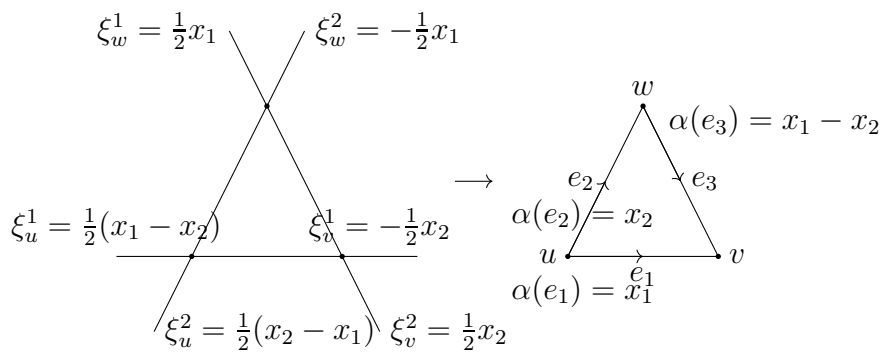

\section*{Acknowledgement}
We would like to thank the anonymous referee for telling us that our previous proof of the main theorem can be greatly reduced and also pointed out Proposition~\ref{realization}.


\end{document}